\documentclass[11pt]{article}
\usepackage{amsmath,amssymb,amsthm, epsfig}
\usepackage{amsfonts}
\usepackage{amsmath}
\usepackage{amssymb}
\usepackage{color}
\usepackage[mathscr]{eucal}

\title{Sharp regularity for general Poisson equations \\ with borderline sources}

\author{Eduardo V. Teixeira}

\newlength{\hchng}
\newlength{\vchng}
\def \dist {\mathrm{dist}}

\def \Leb {\mathscr{L}^n}
\newtheorem{theorem}{Theorem}[section]
\newtheorem{lemma}[theorem]{Lemma}

\theoremstyle{definition}

\theoremstyle{remark}
\newtheorem{remark}[theorem]{Remark}
\numberwithin{equation}{section}

\newcommand{\intav}[1]{\mathchoice {\mathop{\vrule width 6pt height 3 pt depth  -2.5pt
\kern -8pt \intop}\nolimits_{\kern -6pt#1}} {\mathop{\vrule width
5pt height 3  pt depth -2.6pt \kern -6pt \intop}\nolimits_{#1}}
{\mathop{\vrule width 5pt height 3 pt depth -2.6pt \kern -6pt
\intop}\nolimits_{#1}} {\mathop{\vrule width 5pt height 3 pt depth
-2.6pt \kern -6pt \intop}\nolimits_{#1}}}

\begin{document}

\maketitle

\begin{abstract}
This article concerns optimal estimates for non-homogeneous
degenerate elliptic equation with source functions in borderline
spaces of integrability. We deliver sharp H\"older continuity
estimates for solutions to $p$-degenerate elliptic equations in
rough media with sources in the weak Lebesgue space
$L_\text{weak}^{\frac{n}{p} + \epsilon}$.  For the borderline
case, $f \in L_\text{weak}^{\frac{n}{p}}$, solutions may not be
bounded; nevertheless we show that solutions have bounded mean
oscillation, in particular  John-Nirenberg's exponential
integrability estimates can be employed. All the results presented
in this paper are optimal. Our approach is inspired by a powerful
Caffarelli-type compactness method and it can be employed in a
number of other situations.
\medskip



\end{abstract}

\tableofcontents

\section{Introduction}

Central theme in the theory of elliptic partial differential equations, the classical Poisson equation
\begin{equation}\label{Int - PE}
    -\Delta u = f(X),
\end{equation}
models important problems from theoretical physics, mechanical
engineering to biology, economics, among many other
applications. One of the key objectives in the analysis of Poisson
equations is to assure regularity of $u$ based on smoothness
or integrability properties of its laplacian, $f$. In this context, Schauder
estimates is a fundamental result. It assures that the Hessian of
$u$, $D^2u$, is as regular as $f$, provided $f$ has an appropriate
modulus of continuity. More precisely, if $f \in C^\alpha(B_1)$,
$0 < \alpha < 1$ then $u \in C^{2, \alpha}(B_{1/2})$ and
\begin{equation}\label{Int - Schauder Est}
    \|u\|_{C^{2, \alpha}(B_{1/2})} \le C_n \left \{ \|f\|_{C^{\alpha}(B_{1})} + \|u\|_{L^\infty(B_1)} \right \},
\end{equation}
for a dimensional constant $C_n$. Schauder estimate is sharp in several ways. Clearly if $u \in C^{2,\alpha}$, then its laplacian is $\alpha$-H\"older continuous. Also if $f$ is merely continuous, one cannot assure $u \in C^2$, nor even $C_{\text{loc}}^{1,1}$ bounds are available. Schauder estimates also fail in the upper extreme, $\alpha = 1$, i.e., if $f \in \text{Lip}$, it is \textit{not true} in general that $u \in C^{2,1}_\text{loc}$.

\par

Establishing regularity of solutions to \eqref{Int - PE} reduces to understanding the behavior of the Newtonian potential of $f$,
 \begin{equation}\label{Int - Newton Pot}
    N_f(X) := \int \dfrac{1}{|X-Y|^{n-2}} f(Y) dY.
\end{equation}
The kernel that appears in \eqref{Int - Newton Pot}, $\Gamma(X) = |X-Y|^{2-n}$, is the fundamental solution of the laplacian. The second derivative of $\Gamma$,  $D_{ij} \Gamma \sim |X-Y|^{-n}$ is  not integrable, but it is almost integrable, in the sense that $|X-Y|^\epsilon D_{ij} \Gamma $ is integrable for any $0 < \epsilon$. This is the key observation that explains why Schauder estimates hold when $f \in C^\alpha$, $0 < \alpha <1$, and it fails when $f$ is merely bounded or  continuous.

\par

In several applications, the source function $f$ is not
continuous, but only $q$-integrable, i.e., $f  \in L^q(B_1)$, for
some $1< q < \infty$. In this case, the corresponding regularity
theory, due to Calder\'on and Zygmund,  asserts that $u \in
W^{2,q}(B_{1/2})$ and

\begin{equation}\label{Int - CZ Est}
    \|u\|_{W^{2, q}(B_{1/2})} \le C_n \left \{ \|f\|_{L^{q}(B_{1})} + \|u\|_{L^q(B_1)} \right \},
\end{equation}

In particular, if $f \in L^\infty$, then $u \in W^{2,q}$ for all $q < \infty$ and by Sobolev embedding, $u \in C^{1,\alpha}$ for any $\alpha < 1$. This type of thesis is usually called \textit{almost optimal regularity result}. Heuristically, for borderline hypotheses, \textit{almost optimal regularity result} is the best one should hope for.

\par

Regularity theory for problems in rough heterogeneous media,  i.e., when governed by elliptic equations with measurable coefficients, is rather more sophisticated, and even for the homogeneous equation
$$
    \nabla \cdot \left ( a_{ij}(X) Du \right ) = 0,
$$
solutions are,  in general, known to be only H\"older continuous. This is the content of De Giorgi, Moser and Nash regularity
theory.  Calder\'on-Zygmung regularity estimates are not available in this setting. In even more complex models, the laplacian in \eqref{Int - PE} is
replaced by further involved nonlinear elliptic operators,
\begin{equation}\label{Eq}
    -\nabla \cdot a(X,Du) = f(X),
\end{equation}
where $a \colon B_1 \times \mathbb{R}^n \to \mathbb{R}^n$ is
$p$-degenerate elliptic vector field.  Throughout this paper we
shall always assume the following standard structural assumption
on the vector field $a$:
\begin{equation}\label{Hyp a}
    \left \{
        \begin{array}{rll}
            |a(X,\xi)| + |\partial_\xi a(X, \xi) |    |\xi| &\le&  \Lambda |\xi|^{p-1}  \\
            \lambda  |\xi_1|^{p-2} |\xi_2|^2 &\le & \langle \partial_\xi a(X, \xi_1)\xi_2, \xi_2 \rangle,
        \end{array}
    \right.
\end{equation}
for positive constants $0 < \lambda \le \Lambda < +\infty$.  As
usual in the literature, we could also include a parameter $s \ge
0$ as to distinguish the model $p$-Laplacian operator ($s=0$) from
the nondegenerate one ($s>0$), see for instance \cite{M01, M02}.
Throughout this paper, constants that depend only upon $n$, $p$,
$\lambda$ and $\Lambda$ will be called \textit{universal}.  

%%%%%

We recall that Equation \eqref{Eq} appears for instance as the
Euler-Lagrange equation of the minimization problem
 $$
    \int F(X, \nabla u) + f(X)u dX \to \min,
 $$
where the variational kernel $F(X, \xi)$ is convex in $\xi$, $F(X,\xi) \sim |\xi|^p$ and $F(X, \lambda \xi) = |\lambda|^p F(X, \xi)$. A typical operator to keep in mind is the $p$-laplace in rough media,
 $$
    -\nabla \cdot (a_{ij}(X) |\nabla u|^{p-2} \nabla u ),
 $$
 where $a_{ij}$ is a bounded, positive definite matrix.

\par

The regularity theory for Equation \eqref{Eq} is nowadays fairly  well established; however it is considerably more subtle than the corresponding linear, uniform elliptic theory. For instance, it is well known that $p$-harmonic functions are locally $C^{1,\alpha}$ for some $\alpha$ that depends on dimension and $p$. The precise value of optimal $\alpha$ is, in general, unknown.

\par

The main goal of the present article is to determine optimal and \textit{almost optimal} regularity estimates for solutions to Equation \eqref{Eq}, based on integrability properties (or more generally on the behavior of the distributional function) of the source $f$. The regularity estimates presented in this paper do not depend much on the concept of weak solution used. Indeed, they can be understood as \textit{a priori} estimates that do not depend on any further regularity property of $f$ or $u$. In the proofs, though, we shall always work with distributional solutions. However the same arguments go through, with no change, if one chooses to use the notion of entropy solutions, see \cite{BBGGPV} or any appropriate approximation scheme. Also we mention that, per our primary motivations, throughout the whole paper we shall only consider the range
\begin{equation}\label{range n, p}
    1 < p < n,
\end{equation}
where $n$ is the dimension and $p$ is the degeneracy exponent of the vector field $a(X, Du)$.

\par

For $L^\infty$-bounds of solutions to Equation \eqref{Eq},  the borderline integrability condition on the source function $f$ is $L^{\frac{n}{p}}$. More precisely, if $f \in L^{ \frac{n}{p}+ \epsilon}$, for any tiny $\epsilon > 0$, solutions are bounded; however one cannot bound the $L^\infty$-norm of $u$ by the $L^{\frac{n}{p}}$ norm of $f$.  The first result we show in this paper, Theorem \ref{main}, is an optimal BMO estimate of solutions with source functions in the weak Lebesgue space $L_\text{weak}^{\frac{n}{p}}$. Under slightly different structure assumptions, a similar result has been obtained by G. Mingione,  Theorem 1.12 in \cite{M01}, as a consequence of potential
analysis considerations (see also \cite{XZ}). Our proof is neither based on potential analysis nor on singular integral considerations. Instead, it is inspired by a powerful compactness type of argument, see \cite{C1}, \cite{CP}, and also \cite{AL01, AL02}. The
case $p=n$, i.e., for the $n$-Laplacian equation,
with $f$ being a finite measure relates to the article \cite{DHM}.
These results could be delivered by our methods as well. We emphasize that in the
case $p=n$, $L^1_\text{weak}$ functions may not define a finite measure. Nevertheless, Theorem \ref{main} provides \textit{a priori} estimates for \textit{a priori} regular solutions. When $f$ is also a measure then this implies an existence and regularity theorem together with known approximation machineries.
\par

As soon as the source function $f$ becomes $(\frac{n}{p} + \epsilon)$-integrable, we show that solutions are in fact continuous. Not only do we show continuity of solution, but actually we provide the precise sharp H\"older exponent of continuity of $u$ based only on the integrability of $f$ and the regularity theory available for $a$-harmonic functions. Once more, the proof of such a result is based on compactness method and explores only the behavior of the distributional function of the source $f$, that is, $f$ needs only to belong to the weak Lebesgue space  $L_\text{weak}^{\theta \cdot \frac{n}{p}}$, $ 1 < \theta < p$. In this case, we show, see Theorem \ref{main 2} and Theorem \ref{main 3}, that
$$
      u \in C^{\min \{ \frac{p}{p-1} \cdot \frac{\theta-1}{\theta}, \alpha_0^{-} \}}_\text{loc},
$$
where $\alpha_0$ is the universal optimal H\"older exponent for solutions to $ -\nabla \cdot a (X, Du) = 0$. Furthermore, we obtain the appropriate \textit{a priori} universal estimate. Such a result brings
important novelties. The first one is the optimal
regularity space $u$ lies. In many applications, for instance in free
boundary and geometric problems, it is important to determine
accurately how fast the solution grow away from its zero level
set. In such a setting, knowing the precise regularity estimate is
crucial for the program.  Example of such problems are equations with singular terms, $-\Delta_p u \sim u^{-\gamma}$, $\gamma > 1$. For these free boundary
geometric problems, solutions are expected to behave like
$|X|^\beta$, near a free boundary point. Thus it is important to
establish regularity estimates where potentials are assumed to
belong to $L_\text{weak}^{\frac{n}{\beta \gamma}}$, but not in the
classical Lebesgue space $L^{\frac{n}{\beta \gamma}}$.
Another important advantage of our approach concerns its
flexibility, which allows further
generalizations, for instance to equations with measure data, to
systems, or even to $p$-degenerate equations in nondivergence
form, $ \mathcal{F}(X,u, \nabla u, D^2u) = f$, where
$\mathcal{F}(X,\xi s, \xi \mathrm{p}, \xi M) \sim
\xi^{p-1}\mathcal{F}(X, s,  \mathrm{p}, M)$, for $\xi > 0$. For
this class of problems, compactness is consequence of Harnack type
inequality as in the original approach in \cite{C1}. When projected to the constant coefficient case, the optimal $C^\alpha$ estimate established in this paper is in accordance to the gradient estimates obtained in \cite{M02, DM} through a powerful and sophisticated nonlinear potential theory. Indeed, for the model equation $-\Delta_p u = f \in L_\text{weak}^{\theta\frac{n}{p}}$, it follows from \cite{M02}, Theorem 1 and \cite{DM}, Theorem 1.1  that $\nabla u \in  L_\text{weak}^{\frac{\theta n (p-1)}{p - \theta}}$, thus by Morrey embedding Theorem, $u \in C^{\frac{p}{p-1} \cdot \frac{\theta-1}{\theta}}$.

\par

\medskip

The paper is organized as follows. In Section \ref{Section
Compactness} we prove a basic compactness Lemma which assures that
if $\|f\|_{L^{\frac{n}{p}}_\mathrm{weak}}$ is small, then there
exists an $\alpha_0$-H\"older continuous function close to $u$ in
$L^p(B_{1/2})$. Section \ref{Section proof main} is devoted to the
proof BMO estimates. In Section \ref{Calpha theory} we address the optimal $C^\alpha$ regularity theory.

\bigskip

\noindent{\bf Acknowledgement.} The author  would like to express
his gratitude to Giuseppe Mingione, for several insightful
comments that benefited a lot the final presentation of this
article.  The author also thanks the anonymous referee for such a careful revision. This work is partially supported by CNPq-Brazil.

\section{Compactness of solutions} \label{Section Compactness}

In this section, we establish  a  compactness result for solutions
to non-homogeneous Poisson Equations \eqref{Eq} that will play a
fundamental role in the proof of Theorem \ref{main} and Theorem
\ref{main 2}. In fact Lemma \ref{compactness} follows as a
consequence of Lemma 3.2 in \cite{DM01}. We include here a proof
for completeness purposes.
\begin{lemma}\label{compactness} Let $u \in W^{1,p}(B_1)$ be a weak solution to \eqref{Eq}, with $\intav{B_1} |u|^p dX \le 1$. Given $\delta > 0$, there exists a $0 < \varepsilon \ll 1$, depending only on $p$, $n$, $\lambda$, $\Lambda$, $\nu$ and $\delta$,  such that if
\begin{equation}\label{Hyp Lemma Compactness}
    \|f\|_{L_\mathrm{weak}^{\frac{n}{p}}(B_1)}  \le \varepsilon,
\end{equation}
then there exists a function $h$ in $B_{1/2}$ satisfying
\begin{equation}\label{Hyp Lemma Compactness Homogeneous}
    -\nabla \cdot \overline{a} (X, \nabla h) = 0, \quad \text{ in } B_{1/2},
\end{equation}
for some vector field $\overline{a}$ satisfying \eqref{Hyp a} with the same ellipticity constants $\lambda$ and $\Lambda$
such that
$$
    \intav{B_{1/2}} |u(X) - h(X)|^p dX  < \delta^p.
$$
\end{lemma}
\begin{proof}
 Let us assume, for the purpose of contradiction, that the thesis of the Lemma fails. If so, there would exist a $\delta_0 > 0$ and sequences
 $$
     u_k \in W^{1,p}(B_1), \quad \text{and} \quad  f_k \in L_\mathrm{weak}^{\frac{n}{p}}(B_1),
 $$
 satisfying
 \begin{equation}\label{prof comp eq 00}
     \intav{B_1} |u_k(X)|^p dX \le 1,
 \end{equation}
 for all $k\ge1$,
\begin{equation}\label{prof comp eq 01}
            -\nabla \cdot a_k (X, \nabla u_k) = f_k \text{ in } B_1,
\end{equation}
where $a_k$ satisfies \eqref{Hyp a} with constants $\lambda$ and
$\Lambda$ and
\begin{equation}\label{prof comp eq 02}
            \|f_k\|_{ L_\mathrm{weak}^{\frac{n}{p}}(B_1)} = \text{o}(1),
\end{equation}
as $k \to 0$; however
\begin{equation}\label{prof comp eq 03}
            \intav{B_{1/2}} |u_k(X) - h(X)|^p dX \ge \delta_0^p,
\end{equation}
for any solution $h$ to the homogeneous problem \eqref{Hyp Lemma Compactness Homogeneous} in $B_{1/2}$ and all $k \ge 1$.

Now by standard Caccioppoli's type energy estimates, see for instance, Theorem 6.5 and Theorem 6.1 in \cite{Giusti} (notice that $\frac{p^{*}}{p^{*} - 1} < \frac{n}{p}$ within the range $1<p<n$), we verify that there exists a constant $C = C(n, \lambda, \Lambda)$ such that
$$
    \int_{B_{1/2}} |\nabla u_k|^p dX \le C,
 $$
 for all $k \ge 1$. Thus, up to a subsequence, there exists a function $u \in W^{1,p}(B_{1/2})$ for which
\begin{equation}\label{prof comp eq 04}
    u_k \rightharpoonup u \text{ in }W^{1,p}(B_{1/2}) \quad \text{ and } \quad u_k \to u \text{ in } L^p(B_{1/2}).
\end{equation}
In addition, in view of \eqref{prof comp eq 01} and \eqref{prof
comp eq 02}, by classical truncation arguments, see for instance
\cite{BM},  we know
\begin{equation}\label{prof comp eq 05}
    \nabla u_k(X) \to  \nabla u(X)  \quad \text{for a.e. } X  \in B_{1/2}.
\end{equation}
Furthermore, by Ascoli Theorem, up to a subsequence, the sequence
of vector fields $a_k(X, \cdot)$ converges locally uniformly to a
vector field $\overline{a}$ satisfying \eqref{Hyp a}. Given a test
function $\phi \in W^{1,p}_0(B_{1/2})$, from \eqref{prof comp eq
02}, \eqref{prof comp eq 04} and \eqref{prof comp eq 05} we have
$$
    \begin{array}{lll}
        \displaystyle \int_{B_{1/2}} \overline{a}(X, D u) \cdot \nabla \phi dX &=&  \displaystyle \int_{B_{1/2}} {a_k}(X, D u_k) \cdot \nabla \phi dX + \text{o}(1) \\
        &=&  \displaystyle \int_{B_{1/2}} f_k \varphi dX + \text{o}(1)  \\
        & = &  \text{o}(1),
    \end{array}
$$
as $k \to \infty$. Since $\phi$ was arbitrary, we conclude $u$ is a solution to the homogeneous equation in $B_{1/2}$. Finally we reach a contradiction in \eqref{prof comp eq 03} for $k \gg 1$. The proof of Lemma \ref{compactness} is concluded.
\end{proof}

\begin{remark}
Arguing as in \cite{DM01}, Lemma 3.2, it is possible to avoid the passage to the limit in the proof of Lemma \ref{compactness}, obtaining therefore a function $h$, solution to the homogeneous equation $\nabla \cdot a(X, \nabla h) = 0$, for the original vector field $a$. For our purposes though, it suffices to obtain an equation within the same universal class of $a$, \eqref{Hyp a}.
\end{remark}
\section{Optimal BMO estimates} \label{Section proof main}

In this section we shall establish optimal \textit{a priori} estimates for solutions to
$$
    -\nabla \cdot a(X, Du) = f \in L_\text{weak}^\frac{n}{p}(B_1),
$$
which corresponds to the lower borderline integrability condition on $f$. In particular, $L^\infty$ bounds cannot be achieved under such a weak hypothesis. We recall that a measurable function $f$ is said to belong to the
weak-$L^p(B_1)$ space, denoted by $L^p_{\text{weak}}(B_1)$, if
there exists a constant $K > 0$ for which
\begin{equation}\label{Def weak Lp}
    \Leb \left ( \{X \in B_1 : |f(X)| > \tau \} \right ) \le \dfrac{K^p}{\tau^p}.
\end{equation}
The infimum of all $K > 0$ for which \eqref{Def weak Lp} holds is
defined to be the weak-$L^p$ norm of $u$ and it is denoted by
$\|u\|_{L^p_{\text{weak}} (B_1)}$. Weak $L^p$ spaces play a
fundamental role in Harmonic Analysis, in particular in the theory
of singular integrals.  It is well known that $L^p \varsubsetneq
L^p_{\text{weak}}$. Also, if $\mathcal{M}$ denotes the
Hardy-Littlewood maximal operator,  then $\mathcal{M}(f) \in
L^1_{\text{weak}}$ provided $f \in L^1$, and such a result is
optimal in the sense that $\mathcal{M}(f) $ may not belong to
$L^1$. This is the main reason for which Calder\'on-Zygmung theory
fails for sources in $L^1$.

To motivate the result of this section, we invite the readers to notice that a careful inference in the kernel from \eqref{Int - Newton Pot} revels a lower borderline condition for the source function $f$. In fact, $\Gamma \in L^r$ for any $r < \frac{n}{n-2}$, but $\Gamma \not \in L^{\frac{n}{n-2}}$. That is, by H\"older inequality,
$$
    N_f \in L^\infty \text{ whenever } f \in L^{\frac{n}{2} + \epsilon},
$$
since $\frac{n}{2}$ is the dual exponent of $\frac{n}{n-2}$. When
$f \in L^{\frac{n}{2}}$, $n \ge 3$, Calder\'on-Zygmund estimate
\eqref{Int - CZ Est} reveals that
\begin{equation}\label{Int app CZ}
    u \in W^{2,\frac{n}{2}} \hookrightarrow W^{1,n} \hookrightarrow L^q,
\end{equation}
for any $ 1< q < \infty.$ That is, it provides an  \textit{almost optimal regularity result}. By a duality argument, one finds out that it is impossible to bound the $L_\text{loc}^\infty$-norm of $u$ by the $L^{\frac{n}{2}}$ norm of $f$.  However, an application of Poincar\'e inequality combined with \eqref{Int app CZ} gives
\begin{equation}\label{Int getting BMO}
    \intav{B_r} |u - {u}_r|^n dX \le C_n \int_{B_r} |\nabla u|^n dX \le C_n \|f\|_{L^{\frac{n}{2}}},
\end{equation}
where, ${u}_r$ denotes the mean of $u$ over $B_r$, i.e.,
${u}_r :=\intav{B_r(X_0)} udY$.

Recall a function $u \in L^1(B_1)$ for which there exists a constant $K > 0$ such that
\begin{equation}\label{Def BMO}
    \intav{B_r(X_0)} { | u - {u}_r} | \le K,
\end{equation}
for every $X_0 \in B_1$ and $0 < r < \dist(X_0, \partial B_1)$, is
said to belong to the $\text{BMO}$ space. The infimum of all $K
> 0$ for which \eqref{Def BMO} holds is defined to be the
$\text{BMO}$-norm of $u$ and it is denoted by
$\|u\|_{\text{BMO}}$.

The BMO space was originally introduced by John and Nirenberg in
\cite{JN}. In that very same paper,  John and Nirenberg  proved
the following fundamental estimate: if  $\|u\|_{\text{BMO}} \le
1$, then there exist positive dimensional constants $\alpha$ and
$\beta$ such that
\begin{equation}\label{Int JN estimate}
    \intav{B_1} e^{\alpha |u - {u}_{1}|} dX \le \beta.
\end{equation}
The original motivation for studying these functions apparently
came from the theory of elasticity, \cite{J}. Interestingly
enough, John-Nirenberg's estimate for BMO functions \eqref{Int JN
estimate} is used by Moser as a key ingredient in his striking
proof of Harnack inequality for divergence form uniform elliptic
equations.  Both Jonh-Nirenberg and Moser works were published
simultaneously in the same issue: \textit{Comm. Pure Appl. Math.} Vol XIV,
back in 1961.

\par

Through the years, BMO space and its analogues have been shown to
enjoy many other properties, with deep  applications in analysis.
For our purposes, it is elucidative to think the BMO
space as the correct substitute for $L^\infty$ as the endpoint of
the $L^p$ spaces as $p \uparrow +\infty$.
\par

In what follows, we will establish the corresponding sharp BMO estimate for solutions to $p$-degenerate elliptic equations

\begin{equation}\label{Eq studied in BMO section}
    -\nabla \cdot a(X, Du) = f \in L^{\frac{n}{p}}_\text{weak}(B_1).
\end{equation}
where $a$ satisfies the standard structural condition \eqref{Hyp a}.

\begin{theorem}\label{main} Let $u \in W^{1,p}(B_1)$ be a  solution to
$$
    -\nabla \cdot a (X,  Du) = f(X).
$$
Assume $a$ satisfies \eqref{Hyp a} and $f \in L_{\text{weak}}^{\frac{n}{p}}(B_1)$. Then $u \in \text{BMO}(B_{1/2})$. Furthermore,
$$
    \|u\|_{\text{BMO}(B_{1/2})} \le  C \left ( \|f\|^{\frac{1}{p-1}}_{L_{\text{weak}}^{\frac{n}{p}}(B_1)} + \|u\|_{L^p(B_1)} \right ),
$$
for a constant $C$ that depends only on $n$ $p$, $\lambda$ and
$\Lambda$.
\end{theorem}

\medskip

In view of the parallel described above to the linear theory,  the
estimate from Theorem \ref{main} should be optimal. Indeed, this
is the case. For instance, say, for $p <n$, if we set $f(X) =
|X|^{-p}$, it is easy to see that $f \in
L^{\frac{n}{p}}_{\text{weak}}$. Solving $\Delta_p u = f$ with
constant boundary data on $\partial B_1$ one finds $u(X) = c_{n,p}
\cdot \ln |X|$, which is in $\text{BMO}$ but not in $L^\infty$.

The proof of Theorem \ref{main} will be based on the compactness result granted in Lemma  \ref{compactness} and  an iterative scheme. Next Lemma is pivotal to our strategy.

\begin{lemma}\label{key lemma} Let $u \in W^{1,p}(B_1)$ be a weak solution to \eqref{Eq},
with $\intav{B_1} |u|^p dX \le 1$. There exist constants $0 <
\varepsilon_0 \ll 1$, $0 < \lambda_0 \ll 1/2$, that depend only on
$n$, $p$, $\lambda$ and $\lambda$, such that if
\begin{equation}\label{Hyp key Lemma}
    \|f\|_{ L_\text{weak}^{\frac{n}{p}}(B_1)}  \le \varepsilon_0,
\end{equation}
then
\begin{equation}\label{Thesis key Lemma}
    \intav{B_{\lambda_0}} |u(X) - (\intav{B_{\lambda_0}} udY) |^p dX  \le 1.
\end{equation}
\end{lemma}

\begin{proof}
 Initially let us recall a general inequality:
\begin{equation}\label{fact}
    \intav{B_{r}} \left |u - \intav{B_r} udY \right |^p dX \le 2^p \intav{B_{r}} \left | u - \gamma \right |^p dX,
\end{equation}
for any $u \in L^p$ and any real number $\gamma$. Indeed,  by triangular inequality,
$$
    \begin{array}{lll}
         \displaystyle \left ( \intav{B_{r}} \left |u - \intav{B_r} udY \right |^p dX \right )^{1/p}  &\le&  \displaystyle  \left ( \intav{B_{r}} \left |u -\gamma   \right |^p dX \right )^{1/p} + \left |\intav{B_r} udY  - \gamma \right | \\
          &\le& \displaystyle  \left ( \intav{B_{r}} \left |u -\gamma   \right |^p dX \right )^{1/p} + \intav{B_r} |u - \gamma| dX \\
          &\le & 2  \displaystyle  \left ( \intav{B_{r}} \left |u -\gamma   \right |^p dX \right )^{1/p}.
    \end{array}
$$

In view of Lemma \ref{compactness}, let $h$ be a solution to the homogeneous equation in $B_{1/2}$ such that
\begin{equation}\label{proof key Lemma Eq01}
    \intav{B_{1/2}} |u(X) - h(X)|^p dX \le \dfrac{7 \lambda_0^{n}}{9 \cdot 2^{2p-1}},
\end{equation}
for $\lambda_0 \ll 1/2$ to be regulated soon. Such a choice will determine $\varepsilon_0$. Notice that \eqref{proof key Lemma Eq01} implies $\intav{B_{1/2}} |h(X)|^p dX \le C$, thus, by regularity theory for homogeneous equation, there exists a constant $C>0$ universal such that
$$
    |h(X) - h(0)|  \le C  |X|^{\alpha_0},
$$
where $C$ that depends only on $n$, $p$, $\lambda$ and $\Lambda$. Next, for $\lambda_0 \ll 1/2$ to be chosen, we estimate
\begin{equation}\label{proof key Lemma Eq03}
    \begin{array}{lll}
        \displaystyle \intav{B_{\lambda_0}} |u(X) - h(0)|^p dX &\le&  2^{p-1} \left ( \displaystyle \intav{B_{\lambda_0}} |u(X) - h(X)|^p dX + \displaystyle                \intav{B_{\lambda_0}} |h(X) - h(0)|^p dX \right ) \\
        & \le & \dfrac{7 }{9 \cdot 2^{p}} + C 2^{p-1}   \cdot  \lambda_0^{\alpha_0 p}.
    \end{array}
\end{equation}
Now we can choose $\lambda_0$, depending on dimension $n$ and $p$, $\lambda$ and $\Lambda$  so small that
\begin{equation}\label{proof key Lemma Eq04}
    C 2^{p-1}   \cdot  \lambda_0^{\alpha_0 p} \le \dfrac{2}{9 \cdot 2^{p}},
\end{equation}
and the proof of Lemma \ref{key lemma} follows from \eqref{proof key Lemma Eq03} and \eqref{fact}.
\end{proof}

\par

\noindent\textit{Proof of Theorem \ref{main}.}
Let $u$ be a weak solution to
$$
    -\nabla \cdot a(X, D u) = f(X), \quad \text{ in } B_1.
$$
The proof starts off with a renormalization. Let $\varepsilon_0$ be the universal constant from Lemma \ref{key lemma}. If we change $u$ by $\kappa u$, with $\kappa \ll 1$, so small that
$$
    \kappa^{p-1} \le \dfrac{\varepsilon_0}{\|f\|_{L_\text{weak}^{\frac{n}{p}}(B_1)}} \quad \text{and} \quad \intav{B_1}{ |\kappa \cdot u|^p} dX  \le 1,
$$
we can assume $u$ and $f$ are under the hypotheses of Lemma \ref{key lemma}. In the sequel, we will show
\begin{equation}\label{induction}
    \intav{B_{\lambda_0^k}} \left | u - c_k \right |^p dX \le 1, \quad \forall k \in \mathbb{N}.
\end{equation}
Here $\lambda_0$ is the universal number from Lemma \ref{key lemma} and $c_k$ denotes the average of $u$ over the ball of radio $\lambda_0^{k}$, i.e.,
$$
    c_k := \intav{B_{\lambda_0^k}} u(X) dX.
$$
We show \eqref{induction} by induction. The case $k=1$ follows directly from Lemma \ref{key lemma}. Assume we have verified \eqref{induction} for $k$. We define the real function $v \colon B_1 \to \mathbb{R}$ by
\begin{equation}\label{proof main Eq01}
    v(X) :=    {u (\lambda^k_0 X) - c_k}
\end{equation}
We also define
\begin{equation}\label{proof main Eq02}
    a_{\lambda_0^k} (X, \xi) := a( {\lambda_0^k} X,  \xi) \quad \text{ and } \quad {L}_{\lambda_0^k} \phi := - \nabla \cdot  a_{\lambda_0^k}(X, D\phi).
\end{equation}
Notice that $a_{\lambda_0^k}$ is also $p$-degenerate elliptic, with the same ellipticity constants as $a$. From the induction assumption, we have
\begin{equation}\label{proof main Eq03.1}
    \intav{B_1} |v(X)|^p dX =   \intav{B_{\lambda_0^k}} |u(Y) - c_k|^p dY \le  1.
\end{equation}
Easily one verifies that
\begin{equation}\label{proof main Eq03}
    \left | {L}_{\lambda_0^k} v (X)\right | \le   {\lambda_0^{kp}} \left | f({\lambda_0^k} X) \right |, \quad \text{ a.e. in } B_1.
\end{equation}
If we label $f_{\lambda_0^k} : =  {\lambda_0^{pk}} \left | f({\lambda_0^k} X) \right |$, a direct computation reveals
\begin{equation}\label{proof main Eq04}
    \begin{array}{lll}
        \Leb \left (\{ X\in B_1 : |f_{\lambda_0^k}| > \tau \} \right ) &=&   \Leb \left (\{ X\in B_1 : |f| \ge \frac{\tau}{ {\lambda_0^{kp}} }   \right ) \cdot {\lambda_0^{-nk}} \\
        &\le & \dfrac{\|f\|^{\frac{n}{p}}_{L_\text{weak}^{\frac{n}{p}}} }{\tau^{\frac{n}{p}}}.
    \end{array}
\end{equation}
That is
\begin{equation}\label{proof main Eq04.1}
    \|f_{\lambda_0^k}\|_{L_\text{weak}^{\frac{n}{p}}} \le  \|f\|_{L_\text{weak}^{\frac{n}{p}}} \le \varepsilon_0.
\end{equation}
We have verified that $v$ is under the hypotheses of Lemma \ref{key lemma}, which assures
\begin{equation}\label{proof main Eq05}
    \intav{B_{\lambda_0}} |v(X) - (\intav{B_{\lambda_0}} vdY) |^p dX = \intav{B_{\lambda^{k+1}_0}} |u(X) - c_{k+1}) |^p dX \le 1.
\end{equation}
This concludes the proof of \eqref{induction}. Finally, given $0<r \ll 1$, let $m \in \mathbb{N}$ be such that
$$
    \lambda_0^{m+1} \le r < \lambda_0^m.
$$
If we label $u_r := \intav{B_r} udY$, we estimate,
$$
    \begin{array}{lll}
        \displaystyle \intav{B_r} |u - u_r|^p dX &\le&2^p \displaystyle \intav{B_r} |u - \lambda_0 u_r|^p dX \\
        &\le & \displaystyle \dfrac{2^p}{\lambda_0} \intav{B_{\lambda_0^m}} |u - c_m|^p dX \\
        & \le & C.
    \end{array}
$$
The proof of Theorem \ref{main} can be now concluded my means of a standard covering argument, which we shall omit here. \qed

\bigskip

We finish up this section by highlighting once more that the strategy used in our reasoning to establish Theorem \ref{main} is indeed quite flexible. It is based on a fine scaling balance between the norm of the source $f$ and the homogeneity of the equation itself. This indicates that similar analysis should be possible to be carried on for equations with measure data, provided the solution already lies in a proper Sobolev space, under the classical diffusion assumption $|f|(B_r) \le C r^{n-p}$, for any ball $B_r$ of radius $r$. For that, though, one needs to revisit the proof of Lemma \ref{compactness} and work under appropriate notion of solutions through truncation. We do not intend to pursue that in this present paper.

%%%%%%%%%
\section{$C^\alpha$ regularity} \label{Calpha theory}
%%%%%%%%%

In this section we turn our attention to optimal regularity estimates to Equation \eqref{Eq} when the source function $f$ lies in  a
slightly better space, say, $f \in L_\text{weak}^{\frac{n}{p} + \epsilon}$. In this case, heuristic scaling methods indicate that weak solutions
should be locally bounded. Indeed,  under slightly stronger assumptions on $f$,
boundedness or even continuity of solutions can be delivered by known methods, for instance through
Serrin's Harnack inequality \cite{S2}. Nevertheless this approach hardly reveals the sharp H\"older exponent of continuity of the solution.

\par

In this section we still work under assumption \eqref{Hyp a}. As we have already invoked, it is classical, see for instance \cite{S2}, that $W^{1,p}$ solutions to the homogeneous equation
\begin{equation} \label{Int - Gen PE Hom}
    -\nabla \cdot a(X,Du) = 0, \quad \text{ in } B_1,
\end{equation}
are $\alpha_0$-H\"older continuous in $B_{1/2}$ and
\begin{equation} \label{Int - Gen PE Hom Estimate}
    \|u\|_{C^{\alpha_0}(B_{1/2})} \le C(n,\lambda,\Lambda, p) \|u\|_{L^p(B_1)}.
\end{equation}
The optimal exponent $\alpha_0$ in \eqref{Int - Gen PE Hom
Estimate} depends only upon dimension, $p$ and ellipticity
constants $\lambda$, and $\Lambda$. In general $\alpha_0 < 1$ and its precise value is unknown.

\begin{theorem}\label{main 2}  Let $u \in W^{1,p}(B_1)$ be a  solution to
\begin{equation}\label{Eq main 2}
    -\nabla \cdot a (X, Du) = f(X)
\end{equation}
Assume \eqref{Hyp a} and $f \in L_{\text{weak}}^{\theta \cdot \frac{n}{p}}(B_1)$, $1 < \theta  < p$. Then $u \in C^\alpha(B_{1/2})$, for
\begin{equation}\label{thm 2 alpha restriction}
    \alpha = \min \{ \frac{p}{p-1} \cdot \frac{\theta-1}{\theta}, \alpha_0^{-} \},
\end{equation}
where $\alpha_0$ is the universal optimal H\"older exponent for solutions to $ -\nabla \cdot a (X, Du) = 0$.
Furthermore,
$$
    \|u\|_{C^\alpha(B_{1/2})} \le  C(n,\lambda, \Lambda, p, \theta) \left ( \|f\|^{\frac{1}{p-1}}_{L_{\text{weak}}^{\theta \cdot \frac{n}{p}}(B_1)} + \|u\|_{L^p(B_1)} \right ).
$$
\end{theorem}

\medskip

The sharp relation in (\ref{thm 2 alpha restriction}) should be read as follows:
\begin{equation}\label{thm 2 alpha restriction explained}
\left |
\begin{array}{llll}
    \text{If} & \frac{p}{p-1} \cdot \frac{\theta-1}{\theta} < \alpha_0 & \text{ then } & u \in C_\text{loc}^{\frac{p}{p-1} \cdot \frac{\theta-1}{\theta}}.  \\
    \text{If} &\frac{p}{p-1} \cdot \frac{\theta-1}{\theta} \ge \alpha_0&  \text{ then } & u \in C^{\alpha}_\text{loc},  \text{ for any } \alpha < \alpha_0.
\end{array}
\right.
\end{equation}

The proof of Theorem \ref{main 2} will be given in subsection \ref{Section proof main 2} below. Optimality of the thesis of Theorem \ref{main 2} can be checked directly by computing in the unit ball, $B_1$
$$
    \Delta_p |X|^{\frac{p}{p-1} \cdot \frac{\theta-1}{\theta}} = c |X|^{-\frac{p}{\theta}} \in L^{\theta \cdot \frac{n}{p}}_\text{weak}.
$$
It is interesting to notice that $|X|^{-\frac{p}{\theta}}$ is not in the
classical Lebesgue space $L^{\theta \cdot \frac{n}{p}}$.
A valuable feature of Theorem \ref{main 2} is the fact
that it provides universal bounds, i.e., H\"older estimates that
depend only on ellipticity and $p$-degeneracy feature of the operator.
This is particularly important in homogenization problems. However, under continuity (or some sort of VMO condition) on  the medium, we can show that solutions to the homonegenous equation
$$
    -\nabla \cdot a (X, Du) = 0,
$$
are $C^\alpha$ for every $\alpha < 1$. Indeed this fact is an immediate consequence of our next Theorem.

\par

In the sequel, we shall slightly improve the thesis of Theorem \ref{main 2}, provided the medium has some sort of continuity property. For simplicity purposes, for the next Theorem, we shall work under classical continuity assumption on the operator $a$ with respect to the $X$ variable. That is, there exists a modulus of continuity $\tau$ such that
\begin{equation}\label{Cont media} \tag{C}
   |a(X, \xi) - a(Y, \xi)| \le \tau(|X-Y|)|\xi|^{p-1}.
\end{equation}

We remark that under the structural assumption \eqref{Hyp a} solutions to the homogeneous, constant coefficient equation
have a priori $C^{1,\epsilon}$ estimates for $X_0 \in B_{1/2}$ fixed. That is
\begin{equation}\label{C1,alpha theory}
    -\nabla a(X_0, Dh) = 0, ~ B_1 \quad \text{implies} \quad \|h\|_{C^{1,\epsilon}(B_{2/3})} \le C(n,p,\lambda, \Lambda) \|h\|_{L^p(B_1)},
\end{equation}
for some $0< \epsilon < 1$ that depends only on $p$, $n$, $\lambda$ and $\Lambda$, see, for instance, \cite{DiBenedetto}.

\begin{theorem}\label{main 3}  Let $u \in W^{1,p}(B_1)$ be a solution to
\begin{equation}\label{Eq main 3}
    -\nabla \cdot a (X, Du) = f(X).
\end{equation}
Assume \eqref{Hyp a}, \eqref{Cont media} and that $f \in L^{\theta \cdot \frac{n}{p}}(B_1)$, $1 < \theta < p$. Then $u \in C^{ \frac{p}{p-1} \cdot \frac{\theta-1}{\theta}}(B_{1/2})$ and furthermore,
$$
    \|u\|_{C^{ \frac{p}{p-1} \cdot \frac{\theta-1}{\theta}}(B_{1/2})} \le  C(n,\lambda, \Lambda, p, \tau, \theta) \left ( \|f\|_{L_{\text{weak}}^{\theta \cdot \frac{n}{p}}(B_1)} + \|u\|_{L^p(B_1)} \right ).
$$
\end{theorem}

\medskip

Before delivering the proofs of Theorem \ref{main 2} and Theorem \ref{main 3}, let us make few comments about our $C^\alpha$ regularity estimates. Initially, as in Theorem \ref{main}, it seems reasonable to establish the same optimal result for measure data $f$, provided
$|f|(B_r) \le C r^{ \frac{\theta n - p}{\theta}}$, for any ball of radius $r$. As for Theorem \ref{main 3}, continuity condition can be greatly relaxed. In fact all we need is a sort of Cordes-Nirenberg type of condition: there exists a universal constant $\delta_\star > 0$ such that
 $$
    |a(X, \xi) - a(0, \xi)| \le  \delta_\star|\xi|^{p-1}.
 $$
 The upper threshold case for continuity theory, $f \in L^n$, is a delicate issue, see \cite{Teix}.  At this point, though, an interesting consequence of  Theorem \ref{main 2} is that solutions to
 $$
    -\nabla a(X, Du) = f \in L^n_\text{weak}(B_1),
 $$
for measurable coefficients equations,  has \textit{almost} the same modulus of continuity as $a$-harmonic functions, i.e., solutions to $-\nabla a(X, Dh) = 0$. That is, if $a$-harmonic functions in $B_1$ are locally $C^{\alpha_0}$, then solutions to $-\nabla a(X, Du) = f \in L^n_\text{weak}(B_1)$ are locally $C^{\beta}$, for any $0 < \beta < \alpha_0$.  The same analysis employed in Theorem \ref{main 3} gives that for equations with continuous coefficients, solutions to $-\nabla a(X, Du) = f \in L^n_\text{weak}(B_1)$ are locally $C^{\beta}$, for any $0 < \beta < 1$.

\subsection{Proof of Theorem \ref{main 2}}\label{Section proof main 2}

We revisit the proof of Lemma \ref{key lemma}. Suppose $\intav{B_1} |u|^p dX \le 1$ and for $q = \theta \cdot p/n$,
$$
    \varepsilon_1 \ge \|f\|_{L_{\text{weak}}^q(B_1)} \ge c_n \|f\|_{L_\text{weak}^{\frac{n}{p}}},
$$
with $\varepsilon_1 > 0$ to be chosen. From Lemma \ref{compactness} there exists a function $h$, solution to
$$
    -\nabla \cdot a(X, D h) = 0, \quad  \text{ in }  B_{1/2}
$$
such that
$$
    \intav{B_{1/2}} |u(X) - h(X)|^p \le \delta_1.
$$
The latter choice for $\delta_1$ determines $\varepsilon_1$ through the compactness Lemma \ref{compactness}. Since $\|h\|_{L^p}$ is under control, the regularity theory for homogeneous equation assures $h \in C^{\alpha_0}(B_{1/3})$ and for a universal constant $C > 0$,
$$
    |h(X) - h(0)| \le C |X|^{\alpha_0}.
$$
We can readily estimate
\begin{equation} \label{thm 2 eq 01}
\begin{array}{lll}
    \displaystyle \intav{B_{\lambda_1}} | u(X) - h(0)|^p dX & \le & \displaystyle 2^{p-1} \left (\intav{B_{\lambda_1}} | u(X) - h(X)|^p dX + \intav{B_{\lambda_1}} | h(X) - h(0)|^p dX \right ) \\
    & \le & 2^{p-1} \delta_1 \lambda_1^{-n} +  2^{p-1} \lambda_1^{p\alpha_0}.
\end{array}
\end{equation}
Now, fixed  $\alpha < \alpha_0$ we can choose $\lambda_1 \ll 1$ universally small so that
\begin{equation} \label{thm 2 eq 02}
    2^{p-1} \lambda_1^{p \alpha_0} \le \dfrac{1}{10} \lambda_1^{\alpha}.
\end{equation}
Once $\lambda_1$ is chosen as indicated above, we select $\delta_1$ (and therefore $\varepsilon_1$) as
\begin{equation} \label{thm 2 eq 03}
    2^{p-1} \delta_1 = \dfrac{9}{10} \lambda_1^{n+\alpha}.
\end{equation}
If we combine \eqref{thm 2 eq 01}, \eqref{thm 2 eq 02} and \eqref{thm 2 eq 03} we conclude that
\begin{equation} \label{thm 2 key lemma}
    \intav{B_{\lambda_1}} | u(X) - h(0)|^p dX \le \lambda_1^{p\alpha},
\end{equation}
provided
\begin{equation} \label{thm 2 key lemma hyp}
    \|f\|_{L_{\text{weak}}^q(B_1)} \le \varepsilon_1,
\end{equation}
for $0 <\varepsilon_1 \ll 1$ that depends only on dimension, $p$ $\lambda$, $\Lambda$ and $\alpha < \alpha_0$. In addition, from the regularity theory for homogeneous equation,
\begin{equation} \label{thm 2 key lemma estimate}
    |h(0)| \le C,
\end{equation}
for a universal constant $C>0$.

We remind that the assumptions  $\intav{B_1} |u|^p dX \le 1  \text{ and } \|f\|_{L_{\text{weak}}^q(B_1)} \le \varepsilon_1$
can be reached by a simple change of scaling and normalization.  Thus, with no loss of generality,  we can work under these hypotheses.

 \par

In the sequel we shall prove that there exists a convergent sequence $\{\mu_k\}_{k\in \mathbb{N}}  \subset \mathbb{R}$ for which
\begin{equation} \label{thm 2 induction}
    \intav{B_{\lambda_1^k}} | u(X) - \mu_k|^p dX \le \lambda_1^{kp\alpha}.
\end{equation}
As before, we will verify \eqref{thm 2 induction} by induction. The case $k =1$ is precisely \eqref{thm 2 key lemma}, with $\mu_1 = h(0)$. Suppose we have checked  \eqref{thm 2 induction} for $k =1, 2, \cdots, m$. Define
\begin{equation} \label{thm 2 def v}
    v(X) := \dfrac{u(\lambda_1^mX) - \mu_m}{\lambda_1^{m \alpha}}.
\end{equation}
With the same notation as in \eqref{proof main Eq02}, we readily verify, as in (\ref{proof main Eq03}), that
\begin{equation} \label{thm 2 eq v}
  \left |  L_{\lambda_1^{\alpha m}} v(X) \right | \le \lambda_1^{m[p - (p-1)\alpha]} \left | f(\lambda_1^{m} X) \right | =: f_m(X).
\end{equation}
One easily estimates, for any $\tau > 0$,
\begin{equation}\label{proof thm 2 sim Eq04}
    \begin{array}{lll}
        \Leb \left (\{ X\in B_1 : |f_m| > \tau \} \right ) &=&   \Leb \left (\{ X\in B_1 : |f| \ge \frac{\tau}{ {\lambda_1^{{m[p - (p-1)\alpha]}}} }   \right ) \cdot {\lambda_1^{-m\cdot n}} \\
        &\le & \dfrac{\|f\|^q_{L_\text{weak}^{q} }}{\tau^{q}} \cdot \left [ \lambda_1^{m[p - (p-1)\alpha]q} \cdot \lambda_1^{-m \cdot n} \right ] \\
        & \le & \varepsilon_1^q,
    \end{array}
\end{equation}
in view of the sharp assumption \eqref{thm 2 alpha restriction}. We have shown that $v$ is entitled to the conclusion in \eqref{thm 2 key lemma}. Let $h_m$ be the solution to the homogeneous problem that is $\sqrt[p]{\delta_1}$-close to $v$ in $B_{1/2}$ in the $L^p$-distance. We label $h_m(0) = t_m$ and, as  in \eqref{thm 2 key lemma estimate}$,
|t_m| < C$ for a universal constant. Applying  \eqref{thm 2 key lemma} to $v$ we find
\begin{equation} \label{thm 2 end eq 01}
        \intav{B_{\lambda_1}} |v(X) - t_m|^p d X \le  \lambda_1^{p \alpha}.
\end{equation}
Rescaling \eqref{thm 2 end eq 01} back yields
\begin{equation} \label{thm 2 end eq 02}
        \intav{B_{\lambda_1^{m+1}}} |u(X) - (\mu_m + \lambda_1^{m \alpha} t_m) |^p d X \le \lambda_1^{p \alpha(m+1)}.
\end{equation}
Therefore, the induction step for \eqref{thm 2 induction} is verified by taking
$$
    \mu_{m+1} := \mu_m + \lambda_1^{m\alpha} t_m.
$$
Indeed $\{\mu_k\}_{k\in \mathbb{N}}$ is a convergent sequence, because we estimate
$$
    |\mu_{k+j} - \mu_{k}| \le C  \dfrac{\lambda_1^{\alpha k}}{1 - \lambda_1^{\alpha}} = \text{o}(1),
$$
as $k \to \infty$. Finally, if we define
$$
    \overline{\mu} := \lim\limits_{k \to \infty} \mu_k,
$$
and $0 < r < 1$ is arbitrary, estimate \eqref{thm 2 induction} gives
$$
    \intav{B_r} | u(X) - \overline{\mu}|^p dX \le C r^{p\alpha};
$$
therefore $u$ is $\alpha$-H\"older continuous at the origin. The proof of Theorem \ref{main 2} follows now via standard covering arguments, which we omit here.  \qed

\subsection{Proof of Theorem \ref{main 3}} \label{Section proof main 3}

For convenience, let us label $q := \theta \cdot \frac{p}{n} > \frac{p}{n}.$
The proof of Theorem \ref{main 3} is based on the following refinement of the Compactness Lemma \ref{compactness}.

\begin{lemma} \label{lemma comp cont media} Let $u \in W^{1,p}(B_1)$ be a weak solution to \eqref{Eq}, with $\intav{B_1} |u|^p dX \le 1$. Given $\delta > 0$, there exists a $0 < \varepsilon \ll 1$, depending on  depending only on $p$, $n$,
$\lambda$, $\Lambda$ and $\delta$ such that if
\begin{equation}\label{Hyp Lemma Compactness Cont Media}
    \|f\|_{L_\text{weak}^q(B_1)}   \le \varepsilon, \quad \text{and} \quad |a(X, \xi) - a(0, \xi)| \le \varepsilon |\xi|^{p-1},
\end{equation}
then there exists a function $h$ in $B_{1/2}$  solution to
\begin{equation}\label{hom eq compact cont media}
    -\nabla \cdot \overline{a} (Dh) = 0, \quad \text{ in } B_{1/2},
\end{equation}
for some constant coefficient vector field $\overline{a}$
satisfying \eqref{Hyp a} with the same ellipticity constants $\lambda$ and $\Lambda$, such that
$$
    \intav{B_{1/2}} |u(X) - h(X)|^p dX  < \delta^p.
$$
\end{lemma}
\begin{proof}
As before, let us assume, searching for a contradiction, that the thesis of the Lemma fails. If so, there would exist a $\delta_0 > 0$ and sequences
 $$
     u_k \in W^{1,p}(B_1), \quad \text{and} \quad  f_k \in L_\text{weak}^q(B_1),
 $$
with
 \begin{equation}\label{prof comp CM eq 00}
     \intav{B_1} |u_k(X)|^p dX \le 1,
 \end{equation}
 for all $k\ge1$,
\begin{equation}\label{prof comp CM eq 01}
            -\nabla \cdot a_k (X, Du_k) = f_k \text{ in } B_1,
\end{equation}
where $a_k$ satisfies \eqref{Hyp a},  and
\begin{equation}\label{prof comp CM eq 02}
            \|f_k\|_{L^q(B_1)} +|\xi|^{1-p}|a_k(X, \xi) - a_k(0, \xi)| = \text{o}(1),
\end{equation}
as $k \to 0$; however
\begin{equation}\label{prof comp CM eq 03}
            \intav{B_{1/2}} |u_k(X) - h(X)|^p dX \ge \delta_0,
\end{equation}
for any solution $h$ to a homogeneous, constant coefficient
equation \eqref{hom eq compact cont media}, in $B_{1/2}$ and all
$k \ge 1$. Reasoning as indicated in the proof of Lemma
\ref{compactness}, we have
$$
    \int_{B_{1/2}} |\nabla u_k|^p dX \le C,
$$
 for all $k \ge 1$. Thus, up to a subsequence, there exists a function $u \in W^{1,p}(B_{1/2})$ for which
\begin{equation}\label{prof comp CM eq 04}
    u_k \rightharpoonup u \text{ in }W^{1,p}(B_{1/2}),  \quad u_k \to u \text{ in } L^p(B_{1/2}), \quad \text{and} \quad  \nabla u_k(X) \to \nabla u(X) \text{   a.e. in } B_{1/2}.
\end{equation}
Also, by Ascoli Theorem, there exists a subsequence under which
$a_{k_j}(0, \cdot) \to \overline{a}(0, \cdot)$ locally uniformly.
Thus, for any $X \in B_{1/2}$,
\begin{equation}\label{prof comp CM eq 05}
    |a_{k_j}(X, \xi) - \overline{a}(0, \xi)| \le |a_{k_j}(X, \xi) - a_{k_j}(0, \xi)| + | a_{k_j}(0, \xi) -
    \overline{a}(0, \xi)| = \text{o}(1),
\end{equation}
that is, $a_{k_j}(X, \cdot) \to \overline{a}(0, \cdot)$ locally
uniformly. Finally, given a test function $\phi \in
W^{1,p}_0(B_{1/2})$, in view of \eqref{prof comp CM eq 02},
\eqref{prof comp CM eq 04} and \eqref{prof comp CM eq 05} we have
$$
    \begin{array}{lll}
        \displaystyle \int_{B_{1/2}} a_k(X, \nabla u_k) \cdot \nabla \phi dX &=&  \displaystyle \int_{B_{1/2}} f_k \varphi dX  \\
        & = &  \displaystyle \int_{B_{1/2}} \overline{a}(0, \nabla u) \cdot \nabla \phi dX + \text{o}(1),
    \end{array}
$$
as $k \to \infty$. Since $\phi$ was arbitrary, we conclude $u$ is a  solution to a constant coefficient equation in $B_{1/2}$. Finally we reach a contradiction in \eqref{prof comp CM eq 03} for $k \gg 1$.
\end{proof}

The main difference between Lemma \ref{compactness} and Lemma \ref{lemma comp cont media} is the fact that the former provides existence of a $C^{\alpha_0}$ function close to $u$ under smallness assumptions on the data. The latter gives a $C^1$ function near $u$ under smallness assumptions that also involve continuity of the medium. Thus, the following version of Lemma \ref{key lemma} can be proven by similar arguments used to establish estimate \eqref{thm 2 key lemma}.

\begin{lemma}\label{key lemma CM}  Let $u \in W^{1,p}(B_1)$ be a weak solution to \eqref{Eq}, with $\intav{B_1} |u|^p dX \le 1$. Given $ \alpha < 1$, there exist constants $0 < \varepsilon_0 \ll 1$, $0 < \lambda_0 \ll 1/2$ and $c_0 \in \mathbb{R}$ such that if
\begin{equation}\label{Hyp key Lemma}
    \|f\|_{L_\text{weak}^q(B_1)}  \le \varepsilon_0 \quad  \text{ and } \quad |a(X, \xi) - a(0, \xi)|  \le
    \varepsilon_0 |\xi|^{p-1},
\end{equation}
then
\begin{equation}\label{Thesis key Lemma 01}
    \intav{B_{\lambda_0}} |u(X) - c_0  |^p dX  \le \lambda_0^{p\alpha}.
\end{equation}
\end{lemma}
\begin{proof}
For $\delta > 0$ to be regulated \textit{a posteriori}, let $h$ be a solution to a constant coefficient equation assured by Lemma \ref{lemma comp cont media}, that is $\delta$-close to $u$ in the $L^p$-norm. From $C^{1,\epsilon}$ regularity theory for constant coefficient equations, \eqref{C1,alpha theory}, there exists a constant $C$ depending only on $n, p, \lambda$ and $\Lambda$ such that
$$
    |h(X) - h(0)| \le C|X|.
$$
Since $\|h\|_{L^p} \le C$, by $L^\infty$ bounds,
$$
    |h(0)| \le C.
$$ 
We now estimate,
$$
    \begin{array}{lll}
        \displaystyle \intav{B_{\lambda_0}} |u(X) - h(0)|^p dX & \le &   2^{p-1} \left ( \displaystyle \intav{B_{\lambda_0}} |h(X) - h(0)|^p dX  +  \intav{B_{\lambda_0}} |u(X) - h(X)|^p dX \right ) \\
        & \le & 2^{p-1} \delta^p \lambda_0^{-n} + 2^{p-1} C\lambda_0^p
    \end{array}
$$
Since $0 < \alpha < 1$, it is possible to select $\lambda_0$ small enough as to assure
$$
     2^{p-1} C\lambda_0^p \le \dfrac{1}{2} \lambda_0^{\alpha p}.
$$
Once selected $\lambda_0$, we set
$$
    \delta := \dfrac{1}{2} \lambda_0^{\frac{n}{p} + \alpha},
$$
which determines the smallness condition $\varepsilon_0$ through the compactness Lemma \ref{lemma comp cont media}.
\end{proof}

Finally, the proof of Theorem
\ref{main 3} follows by the induction argument from Section
\ref{Section proof main 2}, having Lemma \ref{key lemma CM} as its starting basis. We omit the details here.

\bibliographystyle{amsplain, amsalpha}

\begin{thebibliography}{60}



\bibitem{AL01} Avellaneda, M., and Lin, F. H., {\it Compactness methods in the theory of homogenization} Comm. Pure Appl. Math. {\bf 40}, 1987, pp. 803-–847.

\bibitem{AL02} Avellaneda, M., and Lin, F. H., {\it $L^p$ bounds on singular integrals in homogenization.} Comm. Pure Appl. Math. {\bf 44}, 1991, pp. 897--910.

\bibitem{BM} Boccardo, L; Murat, Fran\c cois. {\it Almost everywhere convergence of the gradients of solutions to elliptic and parabolic equations.} Nonlinear Analysis, Theory, Methods and Appications {\bf 19}, No  6, 581--597. 

\bibitem{BBGGPV} P. B\'enilan, L. Boccardo, T. Gallou\"et, R. Gariepy, M. Pierre, and J.L. Vazquez, {\it An L1-theory of existence and uniqueness of solutions of nonlinear elliptic equations}, Ann. Scuola Norm. Sup. Pisa. Cl. Science, Ser. IV {\bf 22} (1995), 241-–273.

\bibitem{C1} Caffarelli, Luis A. \textit{Interior a priori estimates for solutions of fully nonlinear equations.} Ann.
of Math. (2) \textbf{130} (1989), no. 1, 189--213.

\bibitem{CP} Caffarelli, Luis; Peral, I. {\it On $W^{1,p}$ estimates for elliptic equations in divergence form.}
Comm. Pure Appl. Math. {\bf 51} (1998) issue 1, 1--24.

\bibitem{DiBenedetto} DiBenedetto, E. {\it $C^{1+\alpha}$ local regularity of weak solutions of degenerate elliptic equations.} Nonlinear Anal. TMA {\bf 7} (1983), 827--850.


\bibitem{DM01} Duzaar F. and  Mingione G. \textit{Harmonic type approximation lemmas} J. Math. Anal. Appl. {\bf 352} (2009), n0 1, 301--335.

\bibitem{DM} Duzaar F. and  Mingione G. \textit{Gradient estimates via nonlinear potentials.} Amer. J. Math. {\bf
133} (2011), 1093--1149.

\bibitem{DHM} Dolzmann, G.; Hungerb\"uhler,  N. and M\"uller, S. \textit{Uniqueness and maximal regularity for nonlinear elliptic systems of n-Laplace type with measure valued right hand side}.  J. Reine Angew. Math. {\bf 520}, 1-35 (2000)

\bibitem{Giusti} Giusti, E. {\it Direct methods in the calculus of variations.} Singapore: World Scientific. vii, 403 p. (2003). ISBN 981-238-043-4 

\bibitem{J} F. John, {\it Rotation and strain.} Comm. Pure Appl. Math. {\bf 14} (1961), 391–-413.

\bibitem{JN} F. John; L. Nirenberg {\it On functions of bounded mean oscillation} Comm. Pure and Appl. Math. Vol XIV (1961) 415--426.

\bibitem{M01} Mingione G. {\it The Calder\'on-Zygmund theory for elliptic problems with measure
data.} Ann Scu. Norm. Sup. Pisa Cl. Sci. (5) 6 (2007), 195--261.

\bibitem{M02} Mingione G. {\it Gradient estimates below the duality
exponent.} Math. Ann. {\bf 346} (2010), 571--627.

\bibitem{M} J. Moser, {\it On Harnack's Theorem for Elliptic Differential Equations.} Comm. Pure Appl. Math. {\bf 14}
(1961), 577–591.

\bibitem{S2} Serrin, James
\textit{Local behavior of solutions of quasi-linear equations.}
Acta Math. \textbf{111} 1964 247--302.

\bibitem{Teix} Teixeira, Eduardo V. \textit{Universal modulus of continuity for solutions to fully nonlinear elliptic equations.}  Preprint available at {\tt  arxiv.org/abs/1111.2728}

\bibitem{XZ}  Zhong, Xiao \textit{On nonhomogeneous quasilinear elliptic equations.} Dissertation, University of Jyv\"askyl\"a, Jyv\"askyl\"a, 1998. Ann. Acad. Sci. Fenn. Math. Diss. No. 117 (1998), 46 pp.



\end{thebibliography}

\bigskip

\noindent \textsc{Eduardo V. Teixeira} \\
\noindent Universidade Federal do Cear\'a \\
\noindent Departamento de Matem\'atica \\
\noindent Campus do Pici - Bloco 914 \\
\noindent Fortaleza, CE - Brazil 60.455-760 \\
 \noindent \texttt{teixeira@mat.ufc.br}

\end{document}